\documentclass[11pt, a4paper]{amsart}
\usepackage{times}
\usepackage{a4wide}
\usepackage{hyperref}
\usepackage[british]{babel}
\usepackage{enumerate, longtable}
\usepackage{amsmath, amscd, amsfonts, amsthm, amssymb, latexsym, comment, graphicx, color, url}
\usepackage[all]{xypic}
\usepackage[T1]{fontenc}
\usepackage[utf8]{inputenc}
\usepackage{verbatim}

\newtheorem{thm}{Theorem}[section]

\newtheorem{defi}[thm]{Definition}

\newtheorem{prop}[thm]{Proposition}
\newtheorem{cor}[thm]{Corollary}

\newcommand{\donothing}[1]{}

\newcommand{\vect}[2]{
 \left(  \begin{smallmatrix} #1 \\ #2 \end{smallmatrix} \right)}

\newcommand{\CC}{\mathbb{C}}

\newcommand{\ZZ}{\mathbb{Z}}

\newcommand{\exclude}[1]{}

\let\OLDthebibliography\thebibliography
\renewcommand\thebibliography[1]{
  \OLDthebibliography{#1}
  \setlength{\parskip}{0pt}
  \setlength{\itemsep}{0pt plus 0.3ex}
}

\begin{document}


 \title{Fast computation of half-integral weight modular forms}

     \author{Ilker Inam}
     \address{Bilecik Seyh Edebali University, Department of Mathematics, Faculty of Arts and Sciences, 11200 Bilecik, Turkey}
     \email{ilker.inam@bilecik.edu.tr}

     \author{Gabor Wiese}
     \address{University of Luxembourg, Department of Mathematics, Maison du Nombre 6, Avenue de la Fonte, L-4364 Esch-sur-Alzette, Luxembourg}
     \email{gabor.wiese@uni.lu}
   

\begin{abstract}
To study statistical properties of modular forms, including for instance Sato-Tate like problems, it is essential to be able to compute a large number of Fourier coefficients.
In this article, we show that this can be achieved in level 4 for a large range of half-integral weights by making use of one of three explicit bases, the elements of which can be calculated via fast power series operations.
\end{abstract}

\maketitle

\section{Introduction}
Computing modular forms usually means computing a chunk of their Fourier expansions. Whereas a relatively small number of Fourier coefficients suffices to distinguish between two modular forms (of given level and weight), for statistical analyses (such as in Sato-Tate like questions) or for the computation of (special values of) L-functions, one often needs all coefficients up to some huge bound.
In the case of classical modular forms, there exist a number of ways of computing Fourier coefficients, such as the Modular Symbols algorithm (e.g.\ \cite{Stein}, \cite{CAMF}) or via the trace formula (e.g.\ \cite{BeCo}), which work for all weights (at least two) and all levels. All these quite sophisticated algorithms are largely outperformed in the case of level one, but in general weight, by an algorithmic incarnation of the well-known fact that every level one modular form can be written as a polynomial in the two basic Eisenstein series $E_4$ and $E_6$ of weights $4$ and $6$, respectively. To the best of the authors' knowledge, this basic approach is the one allowing the computation of most coefficients in the given generality.

The situation for half-integral weight modular forms is quite similar. There exist general algorithms, working in (most) weights and levels. Many of them (such as the ones implemented in Magma \cite{Magma}) are based on algorithms in integral weights. All such general algorithms for half-integral weight forms known to the authors have severe limitations on the number of coefficients that can be achieved within a reasonable time.

The point of this article is to show that there exist very practical and efficient variants of the classical level one method in half-integral weights. This, again, makes a restriction of the level necessary, namely to $\Gamma_0(4)$, but one retains the full range of half-integral weights.
We present three methods allowing us to reach a huge number of Fourier coefficients, even in reasonably big half-integral weights. These three methods build on the same idea: use fast-to write down power series and perform standard operations on them.

All three methods are based on existing and well-known results from the literature, particularly, the papers of Cohen~\cite{Cohen} and Kohnen~\cite{Ko80}.
More precisely, we present three kinds of bases, which we name the {\em Cohen standard basis}, the {\em Kohnen basis} and the {\em Rankin-Cohen basis}, respectively. Their elements can all be computed by writing down simple power series (representing the standard $\vartheta$-function or Eisenstein series) and performing operations on them, such as addition, multiplication, inversion and derivation.

These three bases thus give rise to quite straight-forward algorithms. For a Magma \cite{Magma} implementation, see~\cite{FB}. A previous version was used for the large scale computations underlying~\cite{IW}.
In our implementation, the computations are performed over the rational numbers. As the numerators and denominators of the coefficients get huge, a lot of time and memory is consumed by operations with huge fractions. A significant speed-up and a significant reduction of the used memory is achieved by working over finite fields or $p$-adic rings with fixed precision.

In order to stress that our approach is extremely fast and efficient, we include here in the introduction some concrete timings and compare them with timings using the standard implementations in Magma \cite{Magma} and Pari/GP \cite{BeCo}. We computed the Fourier coefficients of (four) modular forms forming a basis of the full space of modular forms of weight $13/2$ and level~$4$. Our implementation took 1 second to get to $10^5$, 19 seconds to get to $10^6$ and 317 seconds to get to $10^7$. Pari/GP needed 70 seconds for $10^5$ and already 4069 seconds for $10^6$, we did not run the $10^7$ range. Magma already needed 616 seconds to get to $10^5$ and we did not attempt to go further.

At the end of this note, we also include a simple theoretical analysis of the three algorithms and compare their performance experimentally. One can summarise the findings by stating that in small weights $k + \frac{1}{2}$ with even~$k$, the Rankin-Cohen basis performs best for computing the Kohnen plus-space. In all other cases, the plus-space is best computed by the Kohnen basis. For the computation of the full space, the standard basis always behaves very well.

\subsection*{Acknowledgements}

This work was supported by The Scientific and Technological Research Council of Turkey (TUBITAK) with the project number 118F148. I.I.\ acknowledges partial and complement support by Bilecik Seyh Edebali University research project number 2018-01.BSEU.04-01 and would like to thank the University of Luxembourg for the hospitality in several visits. 

The authors are grateful to Henri Cohen for providing them the initial code for obtaining Hecke eigenforms via Rankin-Cohen brackets in Pari/GP and Winfried Kohnen for interesting discussions. 
Both authors thank the Izmir Institute of Technology, where a huge part of this research was carried out, for its great hospitality. 

Thanks also due to the anonymous referees for their valuable comments and fruitful suggestions, which improved the paper.

\section{Bases of spaces of modular forms}
Let $k$ be a non-negative integer. Denote by $M_k(N)$ the $\CC$-vector space of modular forms of weight $k$ and level~$\Gamma_0(N)$ for a positive integer~$N$.
Furthermore, write $M_{k+1/2}(4)$ for the $\CC$-vector space of modular forms of half-integral weight $k+1/2$ and level~$\Gamma_0(4)$.

\smallskip\noindent{\bf Kohnen plus-space}.
The {\em Kohnen plus-space} $M^+_{k+1/2}(4)$ consists of those $f \in M_{k+1/2}(4)$
such that $a_n(f) = 0$ whenever $(-1)^k n \equiv 2,3 \pmod{4}$. Here $a_n(f)$ denotes the $n$-th Fourier coefficient (for $n \in \ZZ$)
of the Fourier expansion of~$f$ at the standard cusp~$\infty$.
Calculating a basis of the Kohnen plus-space $M^+_{k+1/2}(4)$ from a basis of $M_{k+1/2}(4)$ is a straight forward linear algebra computation.
More precisely, take a basis $f_1,\dots,f_m$ of the full space (with precision~$D$), and, for each $1 \le i \le m$, write the coefficients $a_n(f_i)$ for all $0 \le n < D$ such that $(-1)^k n \equiv 2,3 \pmod{4}$ into a vector $v_i$.
Then take the matrix $M$ of these vectors and compute a basis $b_1,\dots,b_r$ of its kernel.
Then a basis of the Kohnen plus-space (with precision~$D$) is given by
$g_i = \sum_{j=1}^m b_{i,j} f_j$ for $1 \le i \le r$ where $b_{i,j}$ is the $j$-th entry of the vector~$b_i$.

\smallskip\noindent{\bf Integral weight basis.}
We start with integral weight modular forms of the lowest possible level~$1$. The basic result is that the natural embedding
$\CC[E_4,E_6] \to \bigoplus_{k \in \ZZ} M_k(1)$
is an isomorphism of graded algebras, where 
$E_k(q) = \frac{-B_k}{2k} + \sum_{n=1}^\infty \sigma_{k-1}(n) q^n \in M_k(1)$
is the standard Eisenstein series of weight $k \in \ZZ_{\ge 4}$, $q = e^{2\pi i z}$ and $B_k$ is the $k$-th Bernoulli number.
Hence, a $\CC$-basis of the vector space of modular forms of level~$1$ and weight~$k$
is given by $E_4^a \cdot E_6^b$ where $a,b \in \ZZ_{\ge 0}$ run through all possibilities such that $k=4a+6b$.
For more details, see Chapter $8$ and Chapter $10.6$ of \cite{CohenSt}.
From a practical perspective, Eisenstein series are fast to write down, whence the computation of the basis is reduced to the multiplication of power series.
We will describe three generalisations to half-integral weights, based on results from the literature.

\smallskip\noindent {\bf Cohen standard basis.} The above result has the following analog in half-integral weights and smallest possible level~$4$.
Henri Cohen proved in \cite[Proposition 1.1]{Cohen} that the natural embedding
$\CC[\vartheta,F_2] \to \bigoplus_{\ell \in \frac{1}{2}\ZZ} M_\ell(4)$
is an isomorphism of graded algebras, where $\vartheta := \sum_{n \in \ZZ} q^{n^2} = 1 + 2 \sum_{n=1}^\infty q^{n^2} \in M_{1/2}(4)$
is the standard theta series and $F_2 := \sum_{n \ge 1 \textnormal{ odd}} \sigma_1(n) q^n \in M_2(4)$
is the Eisenstein series of weight~$2$ and level~$4$. This immediately leads to the following basis of $M_\ell(4)$ for $\ell \in \frac{1}{2}\ZZ$, which can again be computed by writing down $\vartheta$ and $F_2$ and multiplying power series:

\begin{cor}[Cohen standard basis]
Let $\ell \in \frac{1}{2}\ZZ$. Then the modular forms
\[ \vartheta^a \cdot F_2^b \textnormal{ for all } a,b \in \ZZ_{\ge 0} \textnormal{ such that } \ell = \frac{a}{2} + 2b\]
form a basis of $M_\ell(4)$, which we call the {\em Cohen (standard) basis}.
\end{cor}

\smallskip\noindent{\bf Rankin-Cohen basis.} The next basis to describe is based on Rankin-Cohen brackets. We obtained our motivation from the nice example $\delta(z)$ \cite[p.~177]{KZ81}.
We first recall the definition of the Rankin-Cohen bracket \cite[Def.~5.3.23]{CohenSt}:
Let $f,g$ be two modular forms of weights $k$ and $\ell$, respectively, and let $n \in \ZZ_{\ge 0}$. Then
$[f,g]_n := \sum_{j=0}^n (-1)^j \vect {n+k-1}{j} \vect {n+\ell-1}{n-j} f^{(n-j)} g^{(j)}$
is a modular form of weight~$k+\ell+2n$, where $f^{(i)}$ denotes the $i$-th derivative of~$f$ with respect to $q\frac{d}{dq} = \frac{1}{2 \pi i} \frac{d}{dz}$.

We shall now need more general Eisenstein series. For any positive integer $k$ and any pair of primitive Dirichlet characters $\chi_1,\chi_2$, let
$E_k^{\chi_1,\chi_2}(q) = \frac{-B_k^{\chi_1}}{2k} + \sum_{n=1}^\infty \big( \sum_{0< d \mid n} \chi_1(d) \chi_2(n/d) d^{k-1} \big) q^n$
denote the ($q$-expansion of the) normalised Eisenstein series of weight~$k$ attached to the characters $\chi_1,\chi_2$,
where $B_k^{\chi_1}$ denotes the $k$-th generalised Bernoulli number for the character $\chi_1$.
The level of $E_k^{\chi_1,\chi_2}$ is the product of the conductors of $\chi_1$ and $\chi_2$.

\begin{defi}[Rankin-Cohen basis]
\begin{enumerate}[(a)]
\item Let $k \in \ZZ_{\ge 4}$ be even. Let $d = \dim M_{k+1/2}^+(4)$. If the modular forms
\[ \Phi_{k,n} := [E_{k-2n}(4z),\vartheta(z)]_n\]
for $0 \le n \le d-1$ are linearly independent, we call them the {\em Rankin-Cohen basis} of $M_{k+1/2}^+(4)$.
\item Let $k \in \ZZ_{\ge 3}$ be odd. Let $d = \dim M_{k+1/2}(4)$. If the first $d$ of the following modular forms
\[ \Phi_{k,n,1}:= [E_{k-2n}^{1,\chi}(z),\vartheta(z)]_n,
\Phi_{k,n,2} := [E_{k-2n}^{\chi,1}(z),\vartheta(z)]_n\]
for $0 \le n \le \lceil \frac{d}{2} \rceil$ are linearly independent, we call them the {\em Rankin-Cohen basis} of $M_{k+1/2}(4)$.
\end{enumerate}
\end{defi}

We have been unable to prove that the modular forms in the definition are always linearly independent. However,
this was true in all cases we computed.
It does not seem entirely evident how to write down bases via Rankin-Cohen brackets for the plus-space if $k$ is odd and for the full space if $k$ is even which are as simple as the ones above.
It is straight forward to compute the Rankin-Cohen bases by multiplying and differentiating power series.
The sparseness of $\vartheta$ positively effects the speed of the computation.

\smallskip\noindent{\bf Kohnen basis.}
The Kohnen basis is taken from his fundamental paper~\cite{Ko80}, in which he defines the plus-space.
In order to describe it, for every integer $k \ge 2$, let $H_{k+1/2}$ be the Cohen-Eisenstein series explicitly described in the proof of \cite[Theorem 3.1]{Cohen}
as a linear combination of two linearly independent Eisenstein series in $M_{k+1/2}(4)$.

\begin{prop}[Kohnen basis]
\begin{enumerate}[(a)]
\item Let $k \in \ZZ_{\ge 2}$ be even.
Let $a_0 \in \{0,1,2\}$ satisfy $k \equiv a_0 \pmod{3}$ and put $m = \frac{k-4a_0}{6}-1$.
Then the set consisting of the modular forms
\[\begin{matrix}
E_4^{a_0+3a+1}(4z)\cdot E_6^{m-2a}(4z) \cdot H_{5/2}(z),\;\; E_4^{a_0+3a}(4z) \cdot E_6^{m-2a+1}(4z) \cdot \vartheta(z)  \textnormal{ for } 0 \le a \le \lfloor \frac{m}{2} \rfloor, \\
E_4^{\frac{k}{4}}(4z) \cdot \vartheta(z)  \textnormal{ if } 4 \mid k, \textnormal{ and } E_6^{\frac{k-2}{6}}(4z) \cdot  H_{5/2}(z) \textnormal{ if } 6 \mid (k-2)
\end{matrix}\]
forms a basis of $M_{k+1/2}^+(4)$.
\item Let $k \in \ZZ_{\ge 2}$ be odd.
Let $a_0 \in \{0,1,2\}$ satisfy $k \equiv a_0 \pmod{3}$ and put $m = \frac{k-4a_0-9}{6}$.
Then the set consisting of the modular forms
\[\begin{matrix}
E_4^{a_0+3a+1}(4z) \cdot E_6^{m-2a}(4z) \cdot H_{11/2}(z),\;\; E_4^{a_0+3a}(4z) \cdot E_6^{m-2a+1}(4z) \cdot H_{7/2}(z)  \textnormal{ for } 0 \le a \le \lfloor \frac{m}{2} \rfloor \\
E_4^{\frac{k-3}{4}}(4z) \cdot H_{7/2}(z)  \textnormal{ if } 4 \mid (k-3), \textnormal{ and } 
E_6^{\frac{k-5}{6}}(4z) \cdot H_{11/2}(z) \textnormal{ if } 6 \mid (k-5)
\end{matrix}\]
forms a basis of $M_{k+1/2}^+(4)$.
\end{enumerate}
\end{prop}

\begin{proof}
Let $w = k-6$ if $k$ is even and $w=k-9$ if $k$ is odd.
The modular forms $E_4^{a_0+3a} E_6^{m-2a}$ for $0 \le a \le \lfloor \frac{m}{2} \rfloor$ form a basis of $M_w(1)$.
Multiplying by $E_4$, this space is mapped injectively into $M_{w+4}(1)$ hitting all standard basis elements of the target space
except $E_6^{\frac{w+4}{6}}$ if $6 \mid (w+4)$.
Similarly, multiplying by $E_6$, we obtain a subspace of $M_{w+6}(1)$ containing all standard basis elements except $E_4^{\frac{w+6}{4}}$ if $4 \mid (w+2)$.
Now it suffices to apply \cite[Proposition~1]{Ko80}.
\end{proof}

As in the other cases, the Kohnen bases can be obtained by multiplying power series that can be easily computed.
In particular, we use that Cohen's modular forms $H_{5/2}$, $H_{7/2}$ and $H_{11/2}$ can be explicitly given in terms of the Cohen standard basis (see \cite[Corollary 3.2]{Cohen}).

\section{Complexity and running time}
The three above bases have been implemented as a package FastBases~\cite{FB} in the computer algebra system {\sc Magma}~\cite{Magma}.
In this section, we make some remarks on the complexity and the running times.

The main cost for computing all three bases is the multiplication of power series. We thus ignore the cost of writing down Eisenstein and theta series and differentiating power series. The following table shows the number of multiplications of power series (with fixed precision) in the current implementation as a function of the weight $k \in \frac{1}{2} \ZZ \setminus \ZZ$ (we assume $k-\frac{1}{2}$ even for the Rankin-Cohen basis) with only minor approximations:

\begin{center}
\begin{tabular}{||l||l||l||}
\hline
Cohen Basis of $M_k(4)$ & Kohnen Basis of $M_k(4)^+$ & Rankin-Cohen Basis of $M_k(4)^+$  \\
\hline
\hline
$\frac{3}{2} k + 4$ & $\frac{5}{12} k + 12$ & $\frac{1}{72}k^2 + \frac{3}{4} k + 10$  \\
\hline
\end{tabular}
\end{center}

The number of multiplications for the Cohen and the Kohnen basis is linear in the weight (and the dimension), whereas the dependence is quadratic for the Rankin-Cohen basis. Since all computations are done over the rational numbers, the size of the coefficients also plays a significant role.
According to the Ramanujan-Petersson Conjecture (for instance \cite{Koh94}), $|a_n(f)|$ is asymptotically bounded by $n^{(k-1)/2 + \epsilon}$.
This allows us to give a rough estimate of the running time of the algorithm.
Namely, \cite[Corollary~8.27]{ModernCA} states that the number of word operations for the multiplication of two integral polynomials of degree~$D$ with coefficients of bit size~$S$ is $\tilde{O}(SD)$, where $\tilde{O}$ ignores logarithmic factors. Consequently, the number of word operations in any of our algorithms for the computation of weight $k$ in precision~$D$ is roughly bounded by $\tilde{O}(m(k) \frac{k-1}{2} D)$, where $m(k)$ is the number of power series multiplications listed in the preceding table. This estimate ignores many contributions, such as the (bounded) denominator of the power series, the fact that the power series that we multiply may actually have smaller coefficients, power series such as $\vartheta$ are lacunary, etc.

We remark that it is possible to reduce the number of multiplications in the Kohnen basis by at least a factor of 2 by using the modular function $E_4^3/E_6^2$. As, however, the coefficients of this function are very big, the practially observed running time of the algorithm gets significantly worse if one uses it.

In order to give the reader an idea about actual running times, we include the following short table of running times for given precision on a standard laptop computer\footnote{Intel Core i5 Dual Core CPU 1.80 GHz, 8 GB 1600 MHz DDR3 RAM} using our implementation.

\hspace{1.2cm}

\begin{center}
\begin{tabular}{||l||l|l|l||l|l|l||l|l|l||}
\hline
& \multicolumn{3}{c||}{Cohen Basis of $M_k(4)$} &  \multicolumn{3}{c||}{Kohnen Basis of $M_k(4)^+$} & \multicolumn{3}{c||}{Rankin-Cohen Basis of $M_k(4)^+$}  \\
weight & $10^4$ & $10^5$ & $10^6$ & $10^4$ & $10^5$ & $10^6$ & $10^4$ & $10^5$ & $10^6$\\
\hline
\hline
25/2  & 0.16 & 3.02   & 55.34   & 0.12 & 2.05  & 34.92   &  0.13  & 1.76   & 31.71 \\
\hline
41/2  & 0.40 & 6.86   & 134.19  & 0.22 & 3.39  & 61.96   &  0.25  & 4.08   & 73.91 \\
\hline
101/2 & 2.31 & 50.82  & 934.79  & 0.76 & 17.47 & 283.91  &  2.58  & 51.81  & 796.42  \\
\hline
201/2 & 8.65 & 193.83 & 3486.27 & 3.58 & 60.30 & 1015.79 &  20.63 & 358.83 & 5104.97 \\
\hline
\end{tabular}
\end{center}

A clear conclusion is that the Kohnen basis is the one to choose for the computation of the Kohnen plus-space unless the weights are small, in which case the Rankin-Cohen basis has a slightly better performance due to the lacunarity of~$\vartheta$.
In order to compare with the rough complexity estimate above, we studied the behaviour of the running times experimentally.
First we approximated the running time as a function of the weight $k \in \frac{1}{2}\ZZ \setminus \ZZ$ by $f(k) = b \cdot k^a$  with a particular interest in the exponent~$a$.
The following table shows the calculated exponents~$a$ for the cases of $10^5$ and $10^6$ coefficients.

\hspace{1.2cm}

\begin{center}
\begin{tabular}{||l||l|l|l||}
\hline
nb.~coeff. & Cohen Basis of $M_k(4)$ & Kohnen Basis of $M_k(4)^+$ & Rankin-Cohen Basis of $M_k(4)^+$\\
\hline
\hline
$10^5$ &  1.93  & 1.74 & 2.77  \\
\hline
$10^6$ &  1.93  & 1.81 & 2.57 \\
\hline
\end{tabular}
\end{center}

The results are slightly better than the rough complexity estimate that we gave above. This data makes the advantage of the Kohnen basis over the Rankin-Cohen basis visible.

For fixed modular forms spaces, we also approximated the running time as a function of the number of coefficients~$D$ by the function $g(D) = b \cdot D^a$, again with particular interest in the exponent~$a$.
The following table shows the calculated value of~$a$ for three different weights.\\[-.6cm]

\hspace{1cm}

\begin{center}
\begin{tabular}{||l||l|l|l||}
\hline
weight & Cohen Basis of $M_k(4)$ & Kohnen Basis of $M_k(4)^+$ & Rankin-Cohen Basis of $M_k(4)^+$\\
\hline
\hline
$41/2$  & 1.16 & 1.17 & 1.17 \\
\hline
$101/2$ & 1.31 & 1.21 & 1.22 \\
\hline
$201/2$ & 1.22 & 1.18 & 1.13 \\
\hline
\end{tabular}
\end{center}

We see that the three algorithms present a similar behaviour of the computation time with respect to the number of coefficients, which is also in line with the rough estimate above (which ignores $\log$-factors).
The data suggests a slight advantage for the Rankin-Cohen basis, which might be caused by the lacunarity of $\vartheta$ and its derivatives.

We close the paper with some final remarks.
Often one is interested in Hecke eigenforms. The algorithms of this paper give a basis of the space of modular forms consisting of modular forms with rational coefficients. By basic linear algebra, one can express Hecke eigenforms as linear combinations, usually with coefficients in a number field, of these basis elements. Determining the coefficients in the linear combinations can be done via the Hecke action on $q$-expansions with relatively low precision, so that this computation comes without any extra cost if one is interested in high precision.

The standard and the Kohnen basis rely on an explicit description of a basis in terms of modular forms the $q$-expansion of which can be computed efficiently to a high precision.
We do not know of such a simple description in any higher level, even though it is known that the algebra of modular forms can be described by generators and relations in small levels by a result of Voight and Zureick-Brown \cite[Prop.~11.3.1]{VZB}, proving a conjecture of Rustom~\cite{Rustom}.
Even if only generators (which might have some linear dependence) can be written down quickly to a high precision, similar algorithms as those presented here will be direct consequences. For the moment, this remains an open problem.

\bibliography{References}

\begin{thebibliography}{IDOTW21}

\bibitem[BC18]{BeCo}
Karim Belabas and Henri Cohen.
\newblock Modular forms in {P}ari/{GP}.
\newblock {\em Res. Math. Sci.}, 5(3):Paper No. 37, 19, 2018.

\bibitem[BCP97]{Magma}
Wieb Bosma, John Cannon, and Catherine Playoust.
\newblock The {M}agma algebra system. {I}. {T}he user language.
\newblock volume~24, pages 235--265. 1997.
\newblock Computational algebra and number theory (London, 1993).

\bibitem[Coh75]{Cohen}
Henri Cohen.
\newblock Sums involving the values at negative integers of {$L$}-functions of
  quadratic characters.
\newblock {\em Math. Ann.}, 217(3):271--285, 1975.

\bibitem[CS17]{CohenSt}
Henri Cohen and Fredrik Str\"{o}mberg.
\newblock {\em Modular forms}, volume 179 of {\em Graduate Studies in
  Mathematics}.
\newblock American Mathematical Society, Providence, RI, 2017.
\newblock A classical approach.

\bibitem[IDOTW21]{IW}
\.{I}lker \.{I}nam, Zeynep Demirkol~\"{O}zkaya, Elif Tercan, and Gabor Wiese.
\newblock On the distribution of coefficients of half-integral weight modular
  forms and the {B}ruinier-{K}ohnen conjecture.
\newblock {\em Turkish J. Math.}, 45(6):2427--2440, 1--21 (appendix), 2021.

\bibitem[Koh80]{Ko80}
Winfried Kohnen.
\newblock Modular forms of half-integral weight on {$\Gamma _{0}(4)$}.
\newblock {\em Math. Ann.}, 248(3):249--266, 1980.

\bibitem[Koh94]{Koh94}
Winfried Kohnen.
\newblock On the {R}amanujan-{P}etersson conjecture for modular forms of
  half-integral weight.
\newblock {\em Proc. Indian Acad. Sci. Math. Sci.}, 104(2):333--337, 1994.

\bibitem[KZ81]{KZ81}
W.~Kohnen and D.~Zagier.
\newblock Values of {$L$}-series of modular forms at the center of the critical
  strip.
\newblock {\em Invent. Math.}, 64(2):175--198, 1981.

\bibitem[Rus14]{Rustom}
Nadim Rustom.
\newblock Generators of graded rings of modular forms.
\newblock {\em J. Number Theory}, 138:97--118, 2014.

\bibitem[Ste07]{Stein}
William Stein.
\newblock {\em Modular forms, a computational approach}, volume~79 of {\em
  Graduate Studies in Mathematics}.
\newblock American Mathematical Society, Providence, RI, 2007.
\newblock With an appendix by Paul E. Gunnells.

\bibitem[VZB22]{VZB}
John Voight and David Zureick-Brown.
\newblock The canonical ring of a stacky curve.
\newblock {\em Mem. Amer. Math. Soc.}, 277(1362):v+144, 2022.

\bibitem[vzGG13]{ModernCA}
Joachim von~zur Gathen and J\"{u}rgen Gerhard.
\newblock {\em Modern computer algebra}.
\newblock Cambridge University Press, Cambridge, third edition, 2013.

\bibitem[Wie19]{CAMF}
Gabor Wiese.
\newblock Computational arithmetic of modular forms.
\newblock In {\em Notes from the {I}nternational {A}utumn {S}chool on
  {C}omputational {N}umber {T}heory}, Tutor. Sch. Workshops Math. Sci., pages
  63--170. Birkh\"{a}user/Springer, Cham, 2019.

\bibitem[Wie20]{FB}
Gabor Wiese.
\newblock {\em {{F}ast{B}ases}}, 2020.
\newblock {M}agma package available from
  \url{https://math.uni.lu/wiese/programs/FastBases/}.

\end{thebibliography}
\bibliographystyle{alpha}

\end{document}